\documentclass[reqno,10pt,centertags]{amsart} 
\usepackage{amsmath,amsthm,amscd,amssymb,latexsym,esint,upref,stmaryrd,
enumerate,verbatim,yfonts,bm,mathtools}
\usepackage{color}
\usepackage{hyperref} 
\newcommand*{\mailto}[1]{\href{mailto:#1}{\nolinkurl{#1}}}

\newcommand{\bbN}{{\mathbb{N}}}

\newcommand{\bbQ}{{\mathbb{Q}}}

\newcommand{\bbZ}{{\mathbb{Z}}}

\newcommand{\cA}{{\mathcal A}}
\newcommand{\cB}{{\mathcal B}}

\newcommand{\cH}{{\mathcal H}}

\newcommand{\beq}{\begin{equation}}
\newcommand{\enq}{\end{equation}}

\DeclareMathOperator{\tr}{tr}

\newcommand{\dott}{\,\cdot\,}

\newcommand{\no}{\notag}
\newcommand{\lb}{\label}

\newcommand{\hatt}{\widehat} 
\newcommand{\bi}{\bibitem}

\let\geq\geqslant
\let\leq\leqslant

\makeatletter
\def\theequation{\@arabic\c@equation}

\allowdisplaybreaks 
\numberwithin{equation}{section}

\newtheorem{theorem}{Theorem}[section]

\newtheorem{lemma}[theorem]{Lemma}

\theoremstyle{remark}

\begin{document}

\title[The Product Formula for Regularized Fredholm Determinants]{The Product Formula for Regularized \\Fredholm Determinants} 

\author[T.\ Britz et al.]{Thomas Britz}  
\address{School of Mathematics and Statistics, UNSW, Kensington, NSW 2052,
Australia} 
\email{\mailto{britz@unsw.edu.au}}
\urladdr{\url{https://research.unsw.edu.au/people/dr-thomas-britz}}

\author[]{Alan Carey}  
\address{Mathematical Sciences Institute, Australian National University, 
Kingsley St., Canberra, ACT 0200, Australia 
and School of Mathematics and Applied Statistics, University of Wollongong, NSW, Australia,  2522}  
\email{\mailto{acarey@maths.anu.edu.au}}
\urladdr{\url{http://maths.anu.edu.au/~acarey/}}

\author[]{Fritz Gesztesy}
\address{Department of Mathematics, 
Baylor University, Sid Richardson Bldg., 1410 S.\,4th Street,
Waco, TX 76706, USA}
\email{\mailto{Fritz\_Gesztesy@baylor.edu}}
\urladdr{\url{http://www.baylor.edu/math/index.php?id=935340}}

\author[]{Roger Nichols}
\address{Department of Mathematics, The University of Tennessee at Chattanooga, 
Dept. 6956, 615 McCallie Ave, Chattanooga, TN 37403, USA}
\email{\mailto{Roger-Nichols@utc.edu}}
\urladdr{\url{http://www.utc.edu/faculty/roger-nichols/index.php}}

\author[]{Fedor Sukochev}
\address{School of Mathematics and Statistics, UNSW, Kensington, NSW 2052,
Australia} 
\email{\mailto{f.sukochev@unsw.edu.au}}
\urladdr{\url{https://research.unsw.edu.au/people/scientia-professor-fedor-sukochev}}

\author[]{Dmitriy Zanin} 
\address{School of Mathematics and Statistics, UNSW, Kensington, NSW 2052,
Australia} 
\email{\mailto{d.zanin@unsw.edu.au}}
\urladdr{\url{https://research.unsw.edu.au/people/dr-dmitriy-zanin}}


\date{\today}
\thanks{A.L.C., G.L. and F.S. gratefully acknowledge the support of the Australian Research Council.} 
\thanks{To appear in {\it Proc. Amer. Math. Soc.}}
\@namedef{subjclassname@2020}{\textup{2020} Mathematics Subject Classification}

\begin{abstract} 
For trace class operators $A, B \in \cB_1(\cH)$ ($\cH$ a complex, separable Hilbert space), the product formula for Fredholm determinants holds in the familiar form
\[
{\det}_{\cH} ((I_{\cH} - A) (I_{\cH} - B)) = {\det}_{\cH} (I_{\cH} - A) {\det}_{\cH} (I_{\cH} - B). 
\]
When trace class operators are replaced by Hilbert--Schmidt operators $A, B \in \cB_2(\cH)$ and the Fredholm determinant ${\det}_{\cH}(I_{\cH} - A)$, $A \in \cB_1(\cH)$, by the 2nd regularized Fredholm determinant 
${\det}_{\cH,2}(I_{\cH} - A) =  {\det}_{\cH} ((I_{\cH} - A) \exp(A))$, $A \in \cB_2(\cH)$, the product formula must be replaced by 
\begin{align*} 
{\det}_{\cH,2} ((I_{\cH} - A) (I_{\cH} - B)) &= {\det}_{\cH,2} (I_{\cH} - A) {\det}_{\cH,2} (I_{\cH} - B)   \\
& \quad \times \exp(- \tr_{\cH}(AB)).     
\end{align*}
The product formula for the case of higher regularized Fredholm determinants ${\det}_{\cH,k}(I_{\cH} - A)$, 
$A \in \cB_k(\cH)$, $k \in \bbN$, $k \geq 2$, does not seem to be easily accessible and hence this note aims at filling this gap in the literature. 
\end{abstract}

\maketitle



\section{Introduction} \lb{s1} 

The purpose of this note is to prove a product formula for regularized (modified) Fredholm determinants extending the well-known Hilbert--Schmidt case.

To set the stage, we recall that if $A \in \cB_1(\cH)$ is a trace class operator on the complex, separable Hilbert space $\cH$, that is, the sequence of (necessarily nonnegative) eigenvalues $\lambda_j\big((A^*A)^{1/2}\big)$, $j \in \bbN_0 = \bbN \cup \{0\}$, of $|A| = (A^*A)^{1/2}$ (the singular values of $A$), ordered in nonincreasing magnitude and counted according to their multiplicity, lies in $\ell^1(\bbN_0)$, the Fredholm determinant 
${\det}_{\cH}(I_{\cH} - A)$ associated with $I_{\cH} - A$, $A \in \cB_1(\cH)$, is given by the absolutely convergent infinite product
\begin{equation}
{\det}_{\cH} (I_{\cH} - A) = \prod_{j \in J} [1 -  \lambda_j(A)],    \lb{1.1} 
\end{equation}  
where $\lambda_j(A)$, $j \in J$ (with $J \subseteq \bbN_0$ an approximate index set) are the (generally, complex) eigenvalues of $A$ ordered again with respect to nonincreasing absolute value and now counted according to their algebraic multiplicity. 

A celebrated property of ${\det}_{\cH}(I_{\cH} - \dott)$ that (like the analog of \eqref{1.1}) is shared with the case where $\cH$ is finite-dimensional, is the product formula 
\begin{equation}
{\det}_{\cH} ((I_{\cH} - A) (I_{\cH} - B)) = {\det}_{\cH} (I_{\cH} - A) {\det}_{\cH} (I_{\cH} - B), 
\quad A, B \in \cB_1(\cH)     \lb{1.2} 
\end{equation}
(see, e.g., \cite[pp.~162--163]{GK69}, \cite[Theorem~XIII.105\,$(a)$]{RS78}, \cite[Theorem~3.8]{Si77}, 
\cite[Theorem~3.5\,$(a)$]{Si05}, \cite[Theorem~3.4.10\,$(f)$]{Si15}, \cite[p.~44]{Ya92}).

When extending these considerations to operators $A \in \cB_p(\cH)$, with $ \cB_p(\cH)$, $p \in [1,\infty)$, the 
$\ell^p(\bbN_0)$-based trace ideals (i.e., the eigenvalues $\lambda_j\big((A^*A)^{1/2}\big)$, $j \in \bbN_0$, of 
$(A^*A)^{1/2}$ now lie in $\ell^p(\bbN_0)$, see, e.g., \cite[Sect.~III.7]{GK69}), the $k$th regularized
Fredholm determinant ${\det}_{\cH,k}(I_{\cH} - A)$, $k \in \bbN$, $ k \geq p$, associated with $I_{\cH} - A$, 
$A \in \cB_k(\cH)$, is given by 
\begin{align}
\begin{split} 
{\det}_{\cH,k} (I_{\cH} - A) &= \prod_{j \in J} \bigg([1 -  \lambda_j(A)] \exp\bigg(\sum_{\ell=1}^{k-1} \ell^{-1} 
\lambda_j(A)^{\ell}\bigg)\bigg)      \lb{1.3} \\
&= {\det}_{\cH} \bigg((I_{\cH} - A) \exp\bigg(\sum_{\ell=1}^{k-1} \ell^{-1} A^{\ell}\bigg)\bigg), \quad k \geq p
\end{split} 
\end{align}
(see, e.g., \cite[pp.~1106--1116]{DS88}, \cite[pp.~166--169]{GK69}, \cite{Si77}, \cite[pp.~75--76]{Si05}, 
\cite[pp.~187--191]{Si15}, \cite[p.~44]{Ya92}). In particular, the first line in \eqref{1.3} resembles the structure of canonical infinite product representations of entire functions according to Weierstrass, Hadamard, and Borel (see,
e.g., \cite[Vol. 2, Ch.~II.10]{Ma85}).  

We note that ${\det}_{\cH,k}(I_{\cH} - \dott)$ is continuous on $\cB_{\ell}(\cH)$ for $1 \leq \ell \leq k$, and 
\begin{equation}
{\det}_{\cH,k} (I_{\cH} - AB) = {\det}_{\cH,k} (I_{\cH} - BA),   \quad A, B \in \cB(\cH), \; AB, BA \in \cB_k(\cH)  
\lb{1.4} 
\end{equation}
(this extends to the case where $A$ maps between different Hilbert spaces $\cH_2$ and $\cH_1$ and $B$ from $\cH_1$ to 
$\cH_2$, etc.). 

The analog of the simple product formula \eqref{1.2} no longer holds for $k \geq 2$ and it is well-known in the special Hilbert--Schmidt case $k=2$ that \eqref{1.2} must be replaced by 
\begin{align}
& {\det}_{\cH,2} ((I_{\cH} - A) (I_{\cH} - B)) = {\det}_{\cH,2} (I_{\cH} - A) {\det}_{\cH,2} (I_{\cH} - B) 
\exp(- \tr_{\cH}(AB)),    \no \\
& \hspace*{9cm} A, B \in \cB_2(\cH)     \lb{1.5}
\end{align}
(see, e.g., \cite[p.~169]{GK69}, \cite[p.~76]{Si05}, \cite[p.~190]{Si15}, \cite[p.~44]{Ya92}). Recently, some of us needed the extension of \eqref{1.5} to general $k \in \bbN$, $k \geq 3$, in \cite{CGLNSZ20}, but were not able to find it in the literature; hence, this note aims at closing this gap. 

More precisely, we were interested in a product formula for ${\det}_{\cH,k} ((I_{\cH} - A) (I_{\cH} - B))$ for $A,B \in \cB_k(\cH)$ in terms of 
${\det}_{\cH,k} (I_{\cH} - A)$ and ${\det}_{\cH,k} (I_{\cH} - B)$, $k \in \bbN$, $k \geq 3$. As kindly pointed out to us by Rupert Frank, the particular case where $A$ is a finite rank operator, denoted by $F$, and $B \in \cB_k(\cH)$ was considered in 
\cite[Lemma~1.5.10]{Ha05} (see also, \cite[Proposition~4.8\,$(ii)$]{Ha16}), and the result
\begin{equation}
{\det}_{\cH,k} ((I_{\cH} - F) (I_{\cH} - B)) = {\det}_{\cH} (I_{\cH} - F) {\det}_{\cH,k} (I_{\cH} - B)    
\exp{(\tr_{\cH}(p_n(F,B)))},     \lb{1.6}
\end{equation}
with $p_n(\dott,\dott)$ a polynomial in two variables and of finite rank, was derived. An extension of this formula to three factors, that is, 
\begin{align}
{\det}_{\cH,k} ((I_{\cH} - A)(I_{\cH} - F) (I_{\cH} - B)) &= {\det}_{\cH} (I_{\cH} - F) 
{\det}_{\cH,k} ((I_{\cH} - A)(I_{\cH} - B))    \no \\
& \quad \times \exp{(\tr_{\cH}(p_n(A,F,B)))},     \lb{1.7} 
\end{align}
with $p_n(\dott,\dott,\dott)$ a polynomial in three variables and of finite rank, was derived in \cite[Lemma~C.1]{Fr17}.

The result we have in mind is somewhat different from \eqref{1.6} in that we are interested in a quantitative version of the following fact:  

\begin{theorem} \lb{tD.1} 
Let $k \in \bbN$, and suppose $A, B \in \cB_k(\cH)$. Then
\begin{align}
{\det}_{\cH,k} ((I_{\cH} - A)(I_{\cH} - B)) = {\det}_{\cH,k} (I_{\cH} - A) {\det}_{\cH,k} (I_{\cH} - B) 
\exp(\tr_{\cH}(X_k(A,B))),     \lb{D.1a} 
\end{align}
where $X_k(\dott,\dott) \in \cB_1(\cH)$ is of the form
\begin{align} 
\begin{split} 
X_1(A,B) &= 0,  \\
X_k(A,B) &= \sum_{j_1,\dots,j_{2k-2} = 0}^{k-1} c_{j_1,\dots,j_{2k-2}} 
C_1^{j_1} \cdots C_{2k-2}^{j_{2k-2}}, \quad k \geq 2,
\end{split}
\end{align}
with 
\begin{align}
\begin{split} 
& c_{j_1,\dots,j_{2k-2}} \in \bbQ,   \\
& C_{\ell} = A \text{ or } B, \quad 1 \leq \ell \leq 2k-2,    \\
& k \leq \sum_{\ell=1}^{2k-2} j_{\ell} \leq 2k - 2, \quad k \geq 2.  
\end{split} 
\end{align}
\end{theorem} 

Explicitly, one obtains:
\begin{align}
X_1(A,B) &= 0,   \no \\
X_2(A,B) &= - AB,    \no \\
X_3(A,B) &= 2^{-1} \big[(AB)^2 - AB(A+B) - (A+B)AB\big],     \\
X_4(A,B) &= 2^{-1} (AB)^2 - 3^{-1} \big[AB(A+B)^2+(A+B)^2AB+(A+B)AB(A+B)\big]  \no \\
& \quad + 3^{-1} \big[(AB)^2(A+B)+(A+B)(AB)^2+AB(A+B)AB\big]     \no \\ 
& \quad - 3^{-1} (AB)^3, \no \\
& \hspace*{-9mm} \text{etc.}    \no 
\end{align}

When taking traces (what is actually needed in \eqref{D.1a}), this simplifies to
\begin{align}
\begin{split}
& \tr_{\cH} (X_1(A,B)) = 0,   \\
& \tr_{\cH}(X_2(A,B)) = - \tr_{\cH}(AB),    \\
& \tr_{\cH}(X_3(A,B)) = - \tr_{\cH}\big(ABA + BAB - 2^{-1} (AB)^2\big),     \\
& \tr_{\cH}(X_4(A,B)) = - \tr_{\cH}\big(A^3 B + A^2 B^2 + A B^3 + 2^{-1} (AB)^2    \\
& \hspace*{3.8cm} - (AB)^2A - B(AB)^2 + 3^{-1} (AB)^3\big),    \\
& \quad \text{etc.}
\end{split} 
\end{align}

We present the proof of a quantitative version of Theorem \ref{tD.1} in two parts. In the next section we prove an algebraic result, Lemma \ref{lD.5}, that is the key to the analytic part of the argument
appearing in the final section on regularized determinants.

\section{The Commutator Subspace in the Algebra of \\ Noncommutative Polynomials} \lb{s2}

To prove a quantitative version of Theorem \ref{tD.1} and hence derive a formula for $X_k(A,B)$, we first need to recall some facts on the commutator subspace of an algebra of noncommutative polynomials. 

Let ${\rm Pol_2}$ be the free polynomial algebra in $2$ (noncommuting) variables, $A$ and $B$. Let $W$ be the set of noncommutative monomials (words in the alphabet $\{A,B\}$). (We recall that the set $W$ is a semigroup with respect to concatenation, $1$ is the neutral element of this semigroup, that is, $1$ is an empty word in this alphabet.) Every $x\in{\rm Pol_2}$ can be written as a sum
\begin{equation}
x=\sum_{w\in W}\hatt{x}(w)w.
\end{equation} 
Here the coefficients $\hatt{x}(w)$ vanish for all but finitely many $w\in W$. 

Let $[{\rm Pol_2},{\rm Pol_2}]$ be the commutator subspace of ${\rm Pol_2}$, that is, the linear span of commutators $[x_1,x_2]$, $x_1,x_2\in{\rm Pol_2}$. 

\begin{lemma} \lb{lD.2} 
One has $x\in[{\rm Pol_2},{\rm Pol_2}]$ provided that
\begin{equation} 
\sum_{m=1}^{L(w)}\hatt{x}\bigl(\sigma^m(w)\bigr)=0,\quad w\in W.
\end{equation} 
Here, 
$L(w)$ is the length of each word $w = w_1 w_2\cdots w_{L(w)}$,
$\sigma$ is the cyclic shift given by $\sigma(w)=w_2\cdots w_{L(w)}w_1$.  
\end{lemma}
\begin{proof} One notes that 
\begin{equation} 
x=\sum_{w\in W}\hatt{x}(w)w=\hatt{x}(1)+\sum_{w\neq1} L(w)^{-1}
\sum_{m=1}^{L(w)}\hatt{x}\bigl(\sigma^m(w)\bigr)\sigma^m(w).
\end{equation} 
Obviously, $(\sigma^m(w)-w)\in[{\rm Pol_2},{\rm Pol_2}]$ for each positive integer $m$ and thus,
\begin{equation} 
x\in \bigg(\hatt{x}(1)+\sum_{w\neq1} L(w)^{-1}\sum_{m=1}^{L(w)}\hatt{x}\bigl(\sigma^m(w)\bigr)w+[{\rm Pol_2},{\rm Pol_2}]\bigg).
\end{equation} 
By hypothesis, $\hatt{x}(1)=0$ and
\begin{equation} 
\sum_{m=1}^{L(w)}\hatt{x}\bigl(\sigma^m(w)\bigr)=0,\quad 1\neq w\in W,
\end{equation} 
completing the proof.
\end{proof}

Next, we need some notation. 
Let $k_1,k_2 \in \bbN_0 = \bbN \cup \{0\}$, and set
\begin{equation} 
z_{k_1,k_2}= \begin{cases}
0, & k_1=k_2=0, \\
k_1^{-1} A^{k_1}, & k_1\in \bbN, \, k_2 = 0, \\
k_2^{-1} B^{k_2}, & k_1 = 0, \, k_2 \in \bbN, \\
\sum_{j=1}^{k_1+k_2} j^{-1} 
\sum_{\substack{\pi\in S_j, \, |\pi|=3\\|\pi_1|+|\pi_3|=k_1\\|\pi_2|+|\pi_3|=k_2}}(-1)^{|\pi_3|}z_{\pi}, & k_1,k_2 \in \bbN. 
\end{cases}     \lb{D.z} 
\end{equation} 
Here, $S_j$ is the set of all partitions of the set $\{1,\cdots,j\}$, $1 \leq j \leq k_1 + k_2$. 
(The symbol $|\dott|$ abbreviating the cardinality of a subset of $\bbZ$.) 
The condition $|\pi|=3$ means that $\pi$ breaks the set $\{1,\cdots,j\}$ into exactly $3$ pieces denoted by $\pi_1$, $\pi_2$, and $\pi_3$ (some of them can be empty). The element $z_{\pi}$ denotes the product
\begin{equation}
z_{\pi}=\prod_{m=1}^jz_{m,\pi},\quad z_{m,\pi}=
\begin{cases}
A,& m\in\pi_1,\\
B,& m\in\pi_2,\\
AB,& m\in\pi_3.
\end{cases}
\end{equation} 
Finally, let $W_{k_1,k_2}$ be the collection of all words with $k_1$ letters $A$ and $k_2$ letters $B$. 

Using this notation we now establish a combinatorial fact. 

\begin{lemma} \lb{lD.3} 
Let $k_1,k_2 \in \bbN$. Then  
\begin{equation} 
z_{k_1,k_2}=\sum_{w\in W_{k_1,k_2}}
\Bigg(\sum_{\ell=0}^{n(w)}\frac{(-1)^{\ell}}{k_1+k_2-\ell}\binom{n(w)}{\ell}\Bigg)w, 
\end{equation} 
where 
\begin{equation} 
n(w)=|S(w)|,\quad S(w)= \{1\leq \ell \leq L(w)-1 \, | \,  w_{\ell}=A,\, w_{\ell +1}=B\}.
\end{equation} 
\end{lemma}
\begin{proof} For each $j\in \{1, \ldots, k_1 + k_2\}$, let
\begin{align} 
\Pi_j &= \{ \pi\in   S_j \, | \, |\pi| = 3,\, \, |\pi_1|+|\pi_3| = k_1, \, |\pi_2|+|\pi_3| = k_2\},    \\
\Pi_{j,w} &=\{ \pi\in \Pi_j \, | \, z_\pi = w\}, \quad w\in W_{k_1,k_2}.
\end{align} 
One observes that $|\pi_3|\leq n(w)\leq \min\{k_1,k_2\}$ and that
\begin{equation} 
j = |\pi_1|+|\pi_2|+|\pi_3| = k_1+k_2-|\pi_3|.
\end{equation} 
For any partition $\pi\in\Pi_{j,w}$, let $I\subseteq S(w)$ indicate which subwords $AB$ in $w$ arise from elements in $\pi_3$.  Then $|I|=|\pi_3|=k_1+k_2-j$.  Therefore, each partition in $\pi\in\Pi_{j,w}$ is determined by a unique choice of $I$ and each such choice of $I$ determines the choice of $\pi$ uniquely. This implies that 
\begin{equation} 
|\Pi_{j,w}| = \binom{n(w)}{k_1+k_2-j}.
\end{equation} 
Thus,
\begin{align} 
z_{k_1,k_2} &=\sum_{w\in W_{k_1,k_2}}\sum_{j=1}^{k_1+k_2} j^{-1} \sum_{\pi\in\Pi_{j,w}} (-1)^{|\pi_3|} w 
\no \\
&=\sum_{w\in W_{k_1,k_2}} \sum_{j=1}^{k_1+k_2} (-1)^{k_1+k_2-j} j^{-1} |\Pi_{j,w}|w    \no \\ 
&=\sum_{w\in W_{k_1,k_2}} \sum_{j=1}^{k_1+k_2} (-1)^{k_1+k_2-j} j^{-1} \binom{n(w)}{k_1+k_2-j}w.
\end{align}
Taking into account that
\begin{equation} 
\binom{n(w)}{k_1+k_2-j}=0,\quad k_1+k_2-j\notin\{0,\cdots,n(w)\},
\end{equation} 
it follows that
\begin{align} 
z_{k_1,k_2} &=\sum_{w\in W_{k_1,k_2}} 
\sum_{j=k_1+k_2-n(w)}^{k_1+k_2} (-1)^{k_1+k_2-j} j^{-1} \binom{n(w)}{k_1+k_2-j}w    \no \\
&=\sum_{w\in W_{k_1,k_2}} 
\Bigg(\sum_{\ell=0}^{n(w)}\frac{(-1)^{\ell}}{k_1+k_2-\ell}\binom{n(w)}{\ell}\Bigg)w.
\end{align} 
\end{proof}

We can now prove the main fact about the commutator  subspace of ${\rm Pol_2}$ needed later on.

\begin{lemma} \lb{lD.4} 
For every $k_1,k_2 \in \bbN$, $z_{k_1,k_2}\in[{\rm Pol_2},{\rm Pol_2}]$.  
\end{lemma}
\begin{proof} 
Let $w$ be any element in $W_{k_1,k_2}$ and let $m$ be any positive integer. 
If $\sigma^m(w)$ starts with the subword $AB$, 
then $\sigma^{m+1}(w)$ has the form $B\cdots A$ and therefore has one fewer subwords $AB$ than $\sigma^m(w)$; 
that is, $n\bigl(\sigma^{m+1}(w)\bigr) = n\bigl(\sigma^m(w)\bigr)-1$.
If, however, $\sigma^m(w)$ does not start with the subword $AB$, 
then the $AB$ subwords of $\sigma^{m+1}(w)$ are precisely the $AB$ subwords of $\sigma^m(w)$ each shifted once; 
hence, $n\bigl(\sigma^{m+1}(w)\bigr) = n\bigl(\sigma^m(w)\bigr)$.

Now, to calculate $\sum_{m=1}^{L(w)}\widehat{z_{k_1,k_2}}\bigl(\sigma^m(w)\bigr)$, 
one may assume, by applying cyclic shifts, that $w$ starts with $AB$. 
Then there are $n(w)$ shifted words $\sigma^m(w)$ which start with the subword $AB$, 
and it follows that $n(w)$ of the numbers $\{n\bigl(\sigma^m(w)\bigr)\::\: 1\leq m\leq L(w)\}$ equal $n(w)-1$ 
and that the remaining $L(w)-n(w) = k_1+k_2-n(w)$ numbers equal $n(w)$. 
Lemma~\ref{lD.3} therefore implies that 
\begin{align} 
  \sum_{m=1}^{L(w)}\widehat{z_{k_1,k_2}}\bigl(\sigma^m(w)\bigr)
&=\sum_{m=1}^{L(w)}\Bigg(\sum_{\ell=0}^{n(\sigma^m(w))}\frac{(-1)^{\ell}}{k_1+k_2-\ell}\binom{n\bigl(\sigma^m(w)\bigr)}{\ell}\Bigg)   \no \\
&=n(w) \Bigg(\sum_{\ell=0}^{n(w)-1}\frac{(-1)^{\ell}}{k_1+k_2-\ell}\binom{n(w)-1}{\ell}\Bigg)   \no \\
& \quad +(k_1+k_2-n(w)) \Bigg(\sum_{\ell=0}^{n(w)}\frac{(-1)^{\ell}}{k_1+k_2-\ell}\binom{n(w)}{\ell}\Bigg).
\end{align} 
Since
\begin{equation} 
\binom{n(w)-1}{n(w)}=0,
\end{equation} 
it follows that
\begin{align} 
\sum_{m=1}^{L(w)}\widehat{z_{k_1,k_2}}\bigl(\sigma^m(w)\bigr)
&=n(w) \Bigg(\sum_{\ell=0}^{n(w)}\frac{(-1)^{\ell}}{k_1+k_2-\ell}\binom{n(w)-1}{\ell}\Bigg)    \no \\
& \quad +(k_1+k_2-n(w)) \Bigg(\sum_{\ell=0}^{n(w)}\frac{(-1)^{\ell}}{k_1+k_2-\ell}\binom{n(w)}{\ell}\Bigg)   \no \\
&=\sum_{\ell=0}^{n(w)}\frac{(-1)^{\ell}}{k_1+k_2-\ell} \Bigg(n(w)\binom{n(w)-1}{\ell}    \no \\ 
& \quad +(k_1+k_2-n(w))\binom{n(w)}{\ell}\Bigg).\end{align} 
Clearly,
\begin{equation} 
n(w)\binom{n(w)-1}{\ell}+(k_1+k_2-n(w))\binom{n(w)}{\ell}=(k_1+k_2-\ell)\binom{n(w)}{\ell}, 
\end{equation} 
and thus 
\begin{equation} 
\sum_{m=1}^{L(w)}\widehat{z_{k_1,k_2}}\bigl(\sigma^m(w)\bigr)=\sum_{\ell=0}^{n(w)}(-1)^{\ell}\binom{n(w)}{\ell}=0.  
\end{equation} 
Hence, Lemma~\ref{lD.2} completes the proof. 
\end{proof}

Next, we introduce some further notation. Let $k \in \bbN$ and set
\begin{align} 
\begin{split} 
x_1 &= 0, \\
x_k &=\sum_{j=1}^{k-1} j^{-1} 
\sum_{\substack{\cA \subseteq \{1,\cdots,j\}\\ j+|\cA|\geq k}} 
(-1)^{|\cA|}y_{\cA}, \quad k \geq 2,   \lb{D.33} 
\end{split} \\
\begin{split} 
y_1 &= 0, \\
y_k&=\sum_{j=1}^{k-1} j^{-1} \sum_{\substack{\cA \subseteq \{1,\cdots,j\}\\ j+|\cA|\leq k-1}} 
(-1)^{|\cA|}y_{\cA},   \quad k \geq 2, 
\end{split} \\
y_{\cA}&=\prod_{m=1}^jy_{m,\cA},\quad y_{m,\cA}=
\begin{cases}
A+B,& m \notin \cA, \\
AB,& m \in \cA. 
\end{cases}
\end{align} 
In particular, 
\begin{equation} 
\sum_{j=1}^{k-1} j^{-1} (A+B-AB)^j = x_k + y_k,    \lb{D.binom} 
\end{equation} 
and one notes that the length of the word $y_{\cA}$ subject to $\cA \subseteq \{1,\dots,j\}$, equals 
\begin{equation}
L(y_{\cA}) = \big|\cA^c\big| + 2 |\cA| = j +|\cA|, \quad 1 \leq j \leq k-1, \; k \geq 2    \lb{D.number} 
\end{equation}
(with $A^c = \{1,\dots,j\} \backslash \cA$ the complement of $\cA$ in $\{1,\dots,j\}$). 

Using this notation we can now state the following fact:

\begin{lemma} \lb{lD.5} Let $k \in \bbN$, $k\geq 2$, then 
\begin{equation} 
y_k\in \bigg(\sum_{j=1}^{k-1}\frac1j(A^j+B^j)+[{\rm Pol_2},{\rm Pol_2}]\bigg).
\end{equation} 
\end{lemma}
\begin{proof} Employing 
\begin{equation} 
y_k=\sum_{\substack{k_1,k_2\geq0\\ k_1+k_2 \leq k-1}}z_{k_1,k_2},    \lb{D.a}
\end{equation} 
Lemma \ref{lD.4} yields 
\begin{equation} 
z_{k_1,k_2}\in [{\rm Pol_2},{\rm Pol_2}],\quad k_1,k_2 \in \bbN.   \lb{D.b}
\end{equation} 
Since by \eqref{D.z}, 
\begin{equation} 
z_{0,0}=0, \quad 
z_{k_1,0}= k_1^{-1} A^{k_1},\; k_1\in \bbN,\quad z_{0,k_2}= k_2^{-1} B^{k_2},\; k_2 \in \bbN,    \lb{D.c}
\end{equation} 
combining \eqref{D.a}--\eqref{D.c} completes the proof.
\end{proof}

\section{The Product Formula for $k$th Modified Fredholm Determinants} \lb{s3}

After these preparations we are ready to return to the product formula for regularized determinants and  specialize the preceding algebraic considerations to the context of
Theorem \ref{tD.1}. 

First we recall that by \eqref{D.33} and \eqref{D.number},  
\begin{equation}
x_k =\sum_{j=1}^{k-1} j^{-1} \sum_{\substack{\cA\subseteq \{1,\cdots,j\}\\ j+|\cA|\geq k}}(-1)^{|\cA|}y_{\cA} := X_k(A,B) \in \cB_1(\cH),  
\quad k \geq 2,
\end{equation} 
since for $1 \leq j \leq k-1$, $L(y_{\cA}) = j+|\cA| \geq k$, and hence one obtains the inequality 
\begin{equation} 
\|x_k\|_{\cB_1(\cH)}\leq c_k\max_{\substack{0\leq k_1,k_2<k\\ k_1+k_2\geq k}} 
\|A\|_{\cB_k(\cH)}^{k_1}\|B\|_{\cB_k(\cH)}^{k_2}, \quad k \in \bbN, \; k \geq 2, 
\end{equation} 
for some $c_k > 0$, $k \geq 2$. We also set (cf.\ \eqref{D.33}) $X_1(A,B) = 0$. 

\begin{theorem} Let $k \in \bbN$ and assume that $A,B\in \cB_k(\cH)$. Then 
\begin{equation} 
{\det}_{\cH,k}((I_{\cH}-A)(I_{\cH}-B))={\det}_{\cH,k}(I_{\cH}-A) {\det}_{\cH,k}(I_{\cH}-B) \exp({\tr}_{\cH}(X_k(A,B))).   \lb{D.38A} 
\end{equation} 
\end{theorem}
\begin{proof} First, we suppose that $A,B\in\cB_1(\cH)$. Then it is well-known that
\begin{equation} 
{\det}_{\cH,1}(I_{\cH}-A) {\det}_{\cH,1}(I_{\cH}-B)={\det}_{\cH,1}((I_{\cH}-A)(I_{\cH}-B)),
\end{equation} 
consistent with $X_1(A,B) = 0$. Without loss of generality we may assume that $k \in \bbN$, $k \geq 2$, in 
the following. Employing 
\begin{equation} 
{\det}_{\cH,k}(I_{\cH}-T)={\det}_{\cH}(I_{\cH}-T) \exp\bigg({\tr}_{\cH}\bigg(\sum_{j=1}^{k-1} j^{-1} T^j\bigg)\bigg), 
\quad T \in \cB_1(\cH), 
\end{equation} 
see, for instance, \cite[Lemma~XI.9.22\,(e)]{DS88}, \cite[Theorem~9.2\,(d)]{Si05}, one infers that 
\begin{align} 
&{\det}_{\cH,k}((I_{\cH}-A)(I_{\cH}-B))={\det}_{\cH,k}(I_{\cH}-(A+B-AB))     \no \\
& \quad ={\det}_{\cH}(I_{\cH}-(A+B-AB)) 
\exp\bigg({\tr}_{\cH}\bigg(\sum_{j=1}^{k-1} j^{-1} (A+B-AB)^j\bigg)\bigg)  \no \\
& \quad ={\det}_{\cH}(I_{\cH}-A) {\det}_{\cH}(I_{\cH}-B) 
\exp\bigg({\tr}_{\cH}\bigg(\sum_{j=1}^{k-1} j^{-1} (A+B-AB)^j\bigg)\bigg)    \no \\
& \quad ={\det}_{\cH,k}(I_{\cH}-A) {\det}_{\cH,k}(I_{\cH}-B)   \no \\
& \qquad \times \exp\bigg({\tr}_{\cH}\bigg(\sum_{j=1}^{k-1} j^{-1} \big[(A+B-AB)^j-A^j-B^j\big]\bigg)\bigg).
\end{align} 

By \eqref{D.binom} one concludes that 
\begin{equation} 
{\tr}_{\cH}\bigg(\sum_{j=1}^{k-1} j^{-1} \big[(A+B-AB)^j-A^j-B^j\big]\bigg) 
= {\tr}_{\cH}(x_k) + {\tr}_{\cH}\bigg(y_k-\sum_{j=1}^{k-1} j^{-1} \big(A^j+B^j\big)\bigg).
\end{equation} 
By Lemma \ref{lD.5}, 
\begin{equation} 
y_k-\sum_{j=1}^{k-1} j^{-1} \big(A^j+B^j\big)
\end{equation} 
is a sum of commutators of polynomial expressions in $A$ and $B$. Hence, 
\begin{equation} 
\bigg(y_k-\sum_{j=1}^{k-1} j^{-1} \big(A^j+B^j\big)\bigg) \subset [\cB_1(\cH),\cB_1(\cH)], 
\end{equation} 
and thus, 
\begin{equation} 
{\tr}_{\cH} \bigg(y_k-\sum_{j=1}^{k-1} j^{-1} \big(A^j+B^j\big)\bigg)=0, 
\end{equation} 
proving assertion \eqref{D.38A} for $A,B\in\cB_1(\cH)$.

Since both, the right and left-hand sides in \eqref{D.38A} are continuous with respect to the norm in 
$\cB_k(\cH)$, and $\cB_1(\cH)$ is dense in $\cB_k(\cH)$, \eqref{D.38A} holds for arbitrary 
$A, B \in \cB_k(\cH)$.
\end{proof}


\noindent
{\bf Acknowledgments.} We are indebted to Galina Levitina for very helpful remarks on 
a first draft of this paper and to Rupert Frank for kindly pointing out references \cite{Fr17} and 
\cite{Ha05} to us. We are particularly indebted to the anonymous referee for a very careful reading of our 
manuscript and for making excellent suggestions for improvements.
 
 
\end{document}